\documentclass[a4paper, 11pt]{amsart}
\usepackage{amsmath,amsthm,amssymb,amscd,amsfonts}
\usepackage{enumitem}

\usepackage[backref=page]{hyperref}

\newtheorem{theorem}{Theorem}[section]

\newtheorem{corollary}[theorem]{Corollary}
\theoremstyle{remark}
\newtheorem*{remark}{Remark}

\def\Z{\mathbb Z}
\def\Q{\mathbb Q}
\def\R{\mathbb R}
\def\C{\mathbb C}

\begin{document}

\title{On rigidity of complex Hirzebruch genera on $SU$-manifolds}

\author{G. S. Chernykh}
\date{}
\address{Faculty of Mathematics and Mechanics, Moscow
State University, Russia;\newline
Steklov Mathematical Institute of the Russian Academy of Sciences, Moscow, Russia}
\email{aaa057721@gmail.com}

\thanks{This work was supported by the Russian Science Foundation under grant no. 23-11-00143, https://rscf.ru/en/project/23-11-00143/ . The author is a winner of the Theoretical Physics and Mathematics Advancement Foundation ``BASIS'' program}

\maketitle

\begin{abstract}
We prove that if a complex genus $\varphi \colon \varOmega^U \to R$ is rigid on $SU$-manifolds with a torus action then $\varphi$ is the elliptic Krichever genus.
\end{abstract}

\section*{Introduction}

In \cite{kr90} Krichever defined the generalised elliptic genus $\varphi_{Kr} \colon \varOmega^U_* \to R$ (depending on four parameters) and proved that $\varphi_{Kr}$ is rigid on any manifold with a torus action and an invariant $SU$-structure (a stably complex structure with the first Chern class zero). A complex genus $\varphi \colon \varOmega^U_* \to R$ is rigid on a stably complex manifold $M$ with an action of a torus $T^k$ if its universal $T^k$-equivariant extension 
$$
  \varphi^{T^k} \colon \varOmega^{U:\, T^k}_* \to R[[x_1, x_2,\ldots, x_k]]
$$ 
is constant on the class $[M]\in\varOmega^{U:\, T^k}_*$ (see \cite[\S9.3]{bu-pa15}). Rigidity is an important property of a genus. For example, let $G$ be a connected compact Lie group with the maximal torus $T^k$, $M$ be a stably complex $G$-manifold, and $M\to E\to B$ be a fibre bundle with the fiber $M$ and the structure group $G$. If for any such bundle we have $\varphi([E])=\varphi([M])\varphi([B])$, then the genus $\varphi$ is a $T^k$-rigid on $M$; and conversely, if $\varphi$ is $T^k$-rigid on $M$, then we have $\varphi([E])=\varphi([M])\varphi([B])$ for any bundle with torsion-free $U^*(BG)$ (see \cite[Theorem~9.3.6]{bu-pa15}).

In \cite[Theorem 9.7.13]{bu-pa15} Buchstaber and Panov proved that if a complex genus $\varphi$ is rigid on the sphere $S^6$ (with the $T^2$-invariant almost complex structure from the homogeneous space $G_2/SU(3)=S^6$) and $\varphi([S^6])\ne 0$, then $\varphi$ is the Krichever genus. In particular, $\varphi$ is rigid on any $T^k$-equivariant $SU$-manifold. However, if $\varphi([S^6])=0$, then the rigidity on $S^6$ does not imply that $\varphi$ is the Krichever genus.

The following question was asked in~\cite[Problem~9.7.14]{bu-pa15}: let $\varphi$ be a genus which is rigid on any special unitary $T^k$-manifold, is it true that $\varphi$ is the Krichever genus? In this paper we give an affirmative answer on this question. In fact, we prove that if a complex genus $\varphi$ is rigid on the sphere $S^6$ and on the $10$-dimensional quasitoric $SU$-manifold $\widetilde L(2,3)$ constructed in \cite[Construction 4.9]{lu-pa16}, then $\varphi$ is the Krichever genus. In particular, it is also rigid on all special unitary $T^k$-manifolds.

The author is grateful to Taras Panov for suggesting the problem, fruitful discussions and constant attention to the work.

\section{Preliminaries}

For a smooth compact manifold $M$ (may be with boundary), \emph{a stably complex structure} on $M$ is an equivalence class of real vector bundle isomorphisms $T M \oplus \R^N \cong \xi$, where $T M$ is the tangent bundle of $M$ and $\xi$ is some complex vector bundle over $M$ (considered here as a real vector bundle). The equivalence relation between such isomorphisms is generated by complex isomorphisms $\xi_1\cong \xi_2$ and adding trivial summands $T M \oplus \R^{N+2} \cong \xi\oplus\C$. A \emph{stably complex manifold} is a manifold equipped with a stably complex structure.

A stably complex manifold $M$ is said to be an \emph{$SU$-manifold} if $c_1(TM):=c_1(\xi)=0$.

If $W$ is a stably complex manifold with boundary, then there is the induced stably complex structure on $\partial W$ defined via the isomorphism $T W|_{\partial W} \cong
T M \oplus \R$ (which depends on whether we choose an inward or outward pointing normal vector to $\partial W$ in $W$ as a basis for $\R$, and whether we place this normal vector at the beginning or at the end of the tangent frame of $\partial W$; we need to fix the choice).
If we have a stably complex manifold $M$, then there is the \emph{opposite} stably complex structure on $M$ defined as $T M \oplus \R^N \oplus \R^2 \cong \xi \oplus \C$, where we consider the non-standard isomorphism $\C \to \R^2$, $a+ib \mapsto (a,-b)$. Given a stably complex manifold $M$, we denote by $\overline{M}$ the same manifold with the opposite stably complex structure.

Two closed stably complex manifolds $M_1$ and $M_2$ are said to be \emph{complex bordant} if there is a stably complex manifold with boundary $W$ such that $\partial W = M_1 \sqcup \overline{M_2}$ as stably complex manifolds. This is an equivalence relation on the closed stably complex manifolds. We denote the corresponding equivalence classes, \emph{complex bordism classes}, by $[M]$. The set of complex bordism classes of $n$-dimensional manifolds is denoted by $\varOmega^U_n$. One can easily see that $\partial(M\times I)=M\sqcup \overline{M}$, so the set $\varOmega^U_n$ forms an abelian group with respect to the disjoint union (with the opposite $-[M]=[\overline{M}]$). The graded abelian group $\varOmega^U_* =\bigoplus \varOmega^U_n$ also forms a commutative ring with respect to the Cartesian product.

Let $M$ be a manifold with an action of a torus $T^k$. A stably complex structure on a manifold $M$ is called \emph{$T^k$-invariant} if for any $t\in T^k$ the isomorphism $$\xi \cong T M \oplus \R^N \xrightarrow[\cong]{d t \oplus \mathrm{id}} T M \oplus \R^N \cong \xi$$ is an isomorphism of complex bundles.
One can similarly define the bordism groups $\varOmega^{U: \, T^k}_n$ and the bordism ring $\varOmega^{U:\, T^k}_*$ of such stably complex $T^k$-manifolds.

Let $x \in M$ be an isolated fixed point of the $T^k$-action on $M$. Then we have a complex representation $r_x \colon T^k \to GL(m, \C)$ in the fibre of $\xi$ over $x$, where $m=\dim\xi$. This fibre $\xi_x \cong \C^m$ decomposes as $V\oplus W$, where $r_x$ has no trivial summands on $V$ and is trivial on $W$.
We have $\dim M = \dim_\R V = 2\dim_\C V =2n$ because $x \in M$ is an isolated fixed point (in particular, the dimension of $M$ is necessary even). The nontrivial part $V$ of $r_x$ decomposes into a sum $r_1 \oplus \cdots \oplus r_n$ of one-dimensional complex
$T^k$-representations. In the corresponding coordinates $(z_1,\ldots,z_n) \in V$, an element $t = (e^{2\pi i \varphi_1}, \ldots,e^{2\pi i \varphi_k}) \in T^k$ acts by
$$t\cdot (z_1,\ldots,z_n)=(e^{2\pi i \langle w_1, \varphi \rangle} z_1,\ldots,e^{2\pi i \langle w_n, \varphi \rangle} z_n),$$
where $\varphi = (\varphi_1,\ldots,\varphi_k) \in \R^k$ and $w_j \in \Z^k$, $1\le j \le n$, are the \emph{weights} of the
representation $r_x$ at the fixed point $x$.

Also, the isomorphism $\xi_x\cong T_xM \oplus \R^N$ induces an orientation of the tangent space
$T_xM$, as both $\R^N$ and $\xi_x$ are canonically oriented.
For any fixed point $x \in M$, the \emph{sign} $\sigma(x)$ is $+1$ if the isomorphism
$$T_x M \xrightarrow{\mathrm{id}\oplus 0} T_xM \oplus \R^k \cong \xi_x = V \oplus W \xrightarrow{pr_1} V$$
respects the canonical orientations, and $-1$ if it does not.

Let $P$ be an oriented combinatorial simple $n$-polytope with
$m$ facets, and let $\Lambda$ be an integer $n\times m$-matrix satisfying the following condition: if a vertex $v \in P$ is an intersection of $n$ facets $v = F_{j_1} \cap \cdots \cap F_{j_n}$, then for the square submatrix $\Lambda_v = \Lambda_{j_1,\ldots,j_n}$ formed by the columns $j_1,\ldots,j_n$ of matrix $\Lambda$ we have $\mathrm{det}\,\Lambda_v = \pm1$.

For any such pair $(P,\Lambda)$ one can define an \emph{(omnioriented) quasitoric manifold} $M(P, \Lambda)$, it is a stably complex $2n$-dimensional $T^n$-manifold (see \cite[7.3]{bu-pa15}). The set of fixed points $M^T$ consists of the isolated fixed points which are in bijection with the vertices of $P$. If the fixed point $p$ corresponds to the vertex $v$, then the weights
$w_1(p),\ldots, w_n(p)$ are given by the rows of the square matrix $(\Lambda_v)^{-1}$ (see \cite[Theorem 7.3.18]{bu-pa15}). Also, the sign of the fixed point $p$ is equal to $\mathrm{sign}\bigl(\mathrm{det}(\Lambda_v)\mathrm{det}(a_{j_1} ,\ldots, a_{j_n})\bigr)$, where we denote by $a_j$ the inward-pointing normals to the facets $F_j$ of $P\subset \R^n$ (see \cite[Theorem 7.3.19]{bu-pa15}).

By \emph{a complex genus} we will call any ring homomorphism $\varphi \colon \varOmega^U_*\to R$ from the complex bordism ring to some commutative ring $R$. It's known that such homomorphisms are in bijection with the formal group laws over $R$ (see, for example, \cite[Theorem E.2.6]{bu-pa15}; for a definition of formal group laws and the details on the connection between formal group laws and complex genera see \cite[Appendix E]{bu-pa15}). Let $R$ be a $\Q$-algebra. Then for any formal group law $F(x,y)$ over $R$ there is a unique power series $f(x)=x+\ldots$ over $R$ such that $$f(x+y)=F(f(x),f(y))$$ (see \cite[Theorem E.1.1]{bu-pa15}). Such $f(x)$ is called the \emph{exponent} of $F$. So complex genera with values in a $\Q$-algebra $R$ are in bijection with the power series of the form $x+\ldots \in R[[x]]$. In what follows we will assume that $R$ is a $\Q$-algebra.

The direct relation between a genus $\varphi$ and its exponent $f$ is the following: consider the symmetric function $\prod_{i=1}^n \frac{x_i}{f(x_i)}$, write it as the power series on elementary symmetric functions $P(\sigma_1, \ldots, \sigma_n)$, replacing $\sigma_i$ by the Chern classes $c_i$ we get the characteristic class $P_f=P(c_1, \ldots, c_n)\in H^*(BU; R)$, and then $\varphi([M])=\langle P_f(T M), [M]_\Z \rangle$, where $[M]_\Z$ is the integral fundamental class of $M$ corresponding to the orientation obtained from the stably complex structure (see \cite[Theorem E.3.3]{bu-pa15}).

Conversely, the series $x+\sum_{k\ge 1}\frac{\varphi([\C P^k])}{k+1}x^{k+1}\in R[[x]]$ is the inverse for the series $f(x)$ with respect to the composition of the series (see \cite[Theorem E.2.5]{bu-pa15}).

For any complex genus $\varphi \colon \varOmega^U_* \to R$ one can define its universal \emph{$T^k$-equivariant extension} $\varphi^T \colon \varOmega^{U: \,T^k} \to R[[x_1,\ldots,x_k]]$ (see \cite[9.2-9.3]{bu-pa15}). The constant term of $\varphi^T([M])$ is equal to $\varphi([M])$. A genus $\varphi \colon \varOmega^U_* \to R$ is said to be \emph{$T^k$-rigid} on a stably complex $T^k$-manifold
$M$ whenever $\varphi^T ([M]) \equiv \varphi([M])$, i. e. the power series $\varphi^T([M])\in R[[x_1,\ldots, x_k]]$ is equal to its constant term.

In \cite{bu-pa-ra10} the following localisation formula has been proved.

\begin{theorem}[{\cite[Proposition 6.15]{bu-pa-ra10} or \cite[Theorem 9.4.3]{bu-pa15}}]\label{loc}
  Let $\varphi \colon \varOmega^U \to R$ be a complex genus with the exponent $f(x)=x+\ldots \in R[[x]]$, and let $M$ be a stably complex $2n$-dimensional $T^k$-manifold with isolated fixed points $M^T$. Then the equivariant genus $\varphi^T ([M]) = \varphi([M]) + \ldots$ is given by
  $$\varphi^T([M])=\sum\limits_{p\in M^T} \sigma(p) \prod \limits_{i=1}^n \frac{1}{f(\langle w_i(p), x \rangle)},$$
where $\langle w, x \rangle = w_1x_1 + \ldots + w_kx_k$ for $w = (w_1,\ldots,w_k)$.  
\end{theorem}

\begin{remark}
    The summands on the right hand side of the formula in Theorem \ref{loc} separately belong to the quotient ring of $R[[x_1,\ldots,x_k]]$, but their sum in fact belongs to the ring $R[[x_1,\ldots,x_k]]$ itself.
\end{remark}

So if a genus $\varphi$ is rigid on $M$, then we obtain the following \emph{rigidity equation} on its exponent $f$:

$$c=\sum\limits_{p\in M^T} \sigma(p) \prod \limits_{i=1}^n \frac{1}{f(\langle w_i(p), x \rangle)},$$
where $c=\varphi([M])$.

In \cite{kr90} Krichever defined the generalised elliptic genus, which is now called the \emph{Krichever genus}. We will denote it by $\varphi_{Kr}$. The coefficients of the exponent $f_{Kr}$, viewed as a rational power series, are polynomials in four parameters, so $f_{Kr}\in \Q[\alpha, b_1, b_2, b_3][[x]]$ (and hence $\varphi_{Kr} \colon \varOmega^U_* \to \Q[\alpha, b_1, b_2, b_3]$). More precisely, $f_{Kr}(x)=\frac{e^{\alpha x}}{\Phi(x)}$, where $\Phi(x)\in \Q[b_1, b_2, b_3][[x]]$ is the rational expansion of the elliptic \emph{Baker--Akhiezer function}.

There is a unique power series $f(x) = x+\ldots\in \Q[\delta, \varepsilon][[x]]$ satisfying the differential equation $(f(x)')^2 = 1 - 2\delta f^2(x) + \varepsilon f^4(x)$. Such a power series is known as the Jacobi \emph{elliptic sine} and denoted by $\mathrm{sn}(x)$. It is known that for $b_2=0$ the function $\Phi(x)$ is equal to $\frac{1}{\mathrm{sn}(x)}$, and hence for the corresponding reduction of the Krichever genus we have $f_{Kr}(x)=e^{\alpha x} \mathrm{sn}(x)$. For the details on the Krichever genus and the elliptic sine see \cite[\S E.5]{bu-pa15}.

Krichever proved the following rigidity theorem.

\begin{theorem}[{\cite{kr90}}]\label{krig}
    The Krichever genus $\varphi_{Kr}$ is rigid on $SU$-manifolds.
\end{theorem}

Also, it was proven in \cite{bu-pa-ra10} that the Krichever results imply the following

\begin{theorem}[{\cite[Theorem 6.13]{bu-pa-ra10} or \cite[9.7.11]{bu-pa15}}]\label{zero}
    If $M$ is an omnioriented quasitoric $SU$-manifold, then $\varphi_{Kr}([M])=0$.
\end{theorem}

\section{Main result}

The $6$-dimensional sphere $S^6$ can be obtained as the homogeneous
space $G_2/SU(3)$ for the exceptional Lie group $G_2$. Here $SU(3)$ is the
centraliser of an element of order 3. Furthermore, $S^6 = G_2/SU(3)$ admits a $T^2$-invariant almost complex structure for the maximal torus $T^2 \subset G_2$ (see \cite[Example 9.6.10]{bu-pa15}).

Buchstaber and Panov proved the following result.

\begin{theorem}[{\cite[Theorem 9.7.13]{bu-pa15}}]\label{s6}
    Let $\varphi \colon \varOmega^U_* \to R$ be a complex genus which is rigid on $S^6$ with the described above $T^2$-invariant almost complex structure.
    \begin{enumerate}
\item  If $\varphi([S^6])=c\ne 0$, then $\varphi$ is the Krichever genus with parameter $b_2=c\ne 0$.
\item  If $\varphi([S^6])=0$, then $\varphi$ can be an arbitrary genus with the exponent of the form $f(x) = e^{\alpha x}g(x)$, where $g(x) = x + \ldots$ is an odd power series.
\end{enumerate}
\end{theorem}

In the paper \cite[Construction 4.9]{lu-pa16} Lü and Panov constructed a family $\widetilde L(2k_1, 2k_2+1)$ of omnioriented quasitoric $SU$-manifolds. We will interested in the manifold $\widetilde L(2,3)$.

Define the triangle $\varDelta^2\subset \R^2$ by the inequalities $x\ge 0,\, y \ge0,\, x+y\le1$, and define the tetrahedron $\varDelta^3\subset \R^3$ by the inequalities $x\ge 0,\, y \ge0,\, z\ge 0,\, x+y+z\le1$. Then the polytope $P=\varDelta^2 \times \varDelta^3 \subset \R^5$ is defined by the inequalities $$x_1\ge 0,\, x_2 \ge0,\, x_1+x_2\le1, \,x_3\ge 0,\, x_4 \ge0,\, x_5\ge 0,\, x_3+x_4+x_5\le1$$

Denote the facets of $\varDelta^2\subset\R^2$:

$F_1$ with the inward-pointing normal $(1,0)$,

$F_2$ with the inward-pointing normal $(0,1)$,

and $F_3$ with the inward-pointing normal $(-1,-1)$.

Denote the facets of $\varDelta^3\subset\R^3$:

$E_1$ with the inward-pointing normal $(1,0,0)$,

$E_2$ with the inward-pointing normal $(0,1,0)$,

$E_3$ with the inward-pointing normal $(0,0,1)$,

and $E_4$ with the inward-pointing normal $(-1,-1,-1)$.

Then the polytope $P=\varDelta^2\times\varDelta^3\subset\R^5$ has seven facets: $\widetilde F_1 = F_1 \times \Delta^3$, $\widetilde F_2 = F_2 \times \Delta^3$, $\widetilde F_3 =F_3 \times \Delta^3$, $\widetilde E_1 =\Delta^2 \times E_1$, $\widetilde E_2 =\Delta^2 \times E_2$, $\widetilde E_3 =\Delta^2 \times E_3$, and $\widetilde E_4 =\Delta^2 \times E_4$. The corresponding inward-pointing normals are $a_1=(1,0,0,0,0)$, $a_2=(0,1,0,0,0)$, $a_3=(-1,-1,0,0,0)$, $a_4=(0,0,1,0,0)$, $a_5=(0,0,0,1,0)$, $a_6=(0,0,0,0,1)$, and $a_7=(0,0,-1,-1,-1)$.

Then $\widetilde L(2,3)$ can be defined as the omnioriented quasitoric manifold over the polytope $P=\varDelta^2\times\varDelta^3\subset\R^5$ and with the matrix  $$\Lambda = \begin{pmatrix}
    1&0&1&0&0&0&0\\
    0&1&-1&0&0&0&0\\
    0&0&1&1&0&0&1\\
    0&0&0&0&1&0&-1\\
    0&0&0&0&0&1&1
\end{pmatrix}$$

We have the action of $T^5$ on $\widetilde L(2,3)$ with twelve fixed points corresponding to the vertices of $\varDelta^2\times\varDelta^3$. Any vertex $v$ of $\varDelta^2\times\varDelta^3$ has the form $v_1\times v_2$, where $v_1$ is a vertex of $\varDelta^2$ and $v_2$ is a vertex of $\varDelta^3$. If $v_1=F_{i_1} \cap F_{i_2}$ and $v_2 = E_{j_1} \cap E_{j_2} \cap E_{j_3}$, then $v=v_1\times v_2 =\widetilde F_{i_1} \cap\widetilde F_{i_2}\cap \widetilde E_{j_1} \cap\widetilde E_{j_2} \cap\widetilde E_{j_3}$.

\begin{theorem}\label{main}
    If a complex genus $\varphi \colon \varOmega^U_* \to R$ is rigid on $S^6$, $\varphi([S^6])=0$ and $\varphi$ is rigid on $\widetilde L(2,3)$, then its exponent has the form $f(x)=e^{\alpha x}\mathrm{sn}(x)$, that is, $\varphi$ is the Krichever genus with parameter $b_2=0$.
\end{theorem}

From Theorem \ref{krig}, Theorem \ref{s6} and Theorem \ref{main} we obtain the following

\begin{corollary}
    Let $\varphi \colon \varOmega^U_* \to R$ be a complex genus. Then the following are equivalent:
    \begin{enumerate}
        \item $\varphi$ is the Krichever genus;
        \item $\varphi$ is rigid on $SU$-manifolds;
        \item $\varphi$ is rigid on $S^6$ and $\widetilde L(2,3)$.
    \end{enumerate}
\end{corollary}

\begin{proof}[Proof of Theorem \ref{main}]

Let's consider the weights and signs of fixed points of $\widetilde L(2,3)$.

\medskip

1) $v=(0,0,0,0,0)=\widetilde F_1\cap \widetilde F_2 \cap \widetilde E_1\cap \widetilde E_2\cap \widetilde E_3$

\medskip

$\Lambda_v=\begin{pmatrix}
1 & 0 & 0 & 0 & 0 \\
0 & 1 & 0 & 0 & 0 \\
0 & 0 & 1 & 0 & 0 \\
0 & 0 & 0 & 1 & 0 \\
0 & 0 & 0 & 0 & 1
\end{pmatrix}$ \quad $\Lambda^{-1}_v=\begin{pmatrix}
1 & 0 & 0 & 0 & 0 \\
0 & 1 & 0 & 0 & 0 \\
0 & 0 & 1 & 0 & 0 \\
0 & 0 & 0 & 1 & 0 \\
0 & 0 & 0 & 0 & 1
\end{pmatrix}$

\begin{multline*}
\sigma(v)=\det(\Lambda_v)\det(a_1, a_2, a_4, a_5, a_6)=\\
=\det\begin{pmatrix}
1 & 0 & 0 & 0 & 0 \\
0 & 1 & 0 & 0 & 0 \\
0 & 0 & 1 & 0 & 0 \\
0 & 0 & 0 & 1 & 0 \\
0 & 0 & 0 & 0 & 1
\end{pmatrix} \det\begin{pmatrix}
1 & 0 & 0 & 0 & 0 \\
0 & 1 & 0 & 0 & 0 \\
0 & 0 & 1 & 0 & 0 \\
0 & 0 & 0 & 1 & 0 \\
0 & 0 & 0 & 0 & 1
\end{pmatrix} = 1\cdot 1=1
\end{multline*}

\medskip

2) $v=(1,0,0,0,0)=\widetilde F_2\cap \widetilde F_3 \cap \widetilde E_1\cap \widetilde E_2\cap \widetilde E_3$

\medskip

$\Lambda_v=\begin{pmatrix}
0 & 1 & 0 & 0 & 0 \\
1 & -1 & 0 & 0 & 0 \\
0 & 1 & 1 & 0 & 0 \\
0 & 0 & 0 & 1 & 0 \\
0 & 0 & 0 & 0 & 1
\end{pmatrix}$ \quad $\Lambda^{-1}_v=\begin{pmatrix}
1 & 1 & 0 & 0 & 0 \\
1 & 0 & 0 & 0 & 0 \\
-1 & 0 & 1 & 0 & 0 \\
0 & 0 & 0 & 1 & 0 \\
0 & 0 & 0 & 0 & 1
\end{pmatrix}$

\begin{multline*}
\sigma(v)=\det(\Lambda_v)\det(a_2, a_3, a_4, a_5, a_6)=\\
=\det\begin{pmatrix}
0 & 1 & 0 & 0 & 0 \\
1 & -1 & 0 & 0 & 0 \\
0 & 1 & 1 & 0 & 0 \\
0 & 0 & 0 & 1 & 0 \\
0 & 0 & 0 & 0 & 1
\end{pmatrix} \det\begin{pmatrix}
0 & -1 & 0 & 0 & 0 \\
1 & -1 & 0 & 0 & 0 \\
0 & 0 & 1 & 0 & 0 \\
0 & 0 & 0 & 1 & 0 \\
0 & 0 & 0 & 0 & 1
\end{pmatrix} = (-1)\cdot 1=-1
\end{multline*}

\medskip

3) $v=(0,1,0,0,0)=\widetilde F_1\cap \widetilde F_3 \cap \widetilde E_1\cap \widetilde E_2\cap \widetilde E_3$

\medskip

$\Lambda_v=\begin{pmatrix}
1 & 1 & 0 & 0 & 0 \\
0 & -1 & 0 & 0 & 0 \\
0 & 1 & 1 & 0 & 0 \\
0 & 0 & 0 & 1 & 0 \\
0 & 0 & 0 & 0 & 1
\end{pmatrix}$ \quad $\Lambda^{-1}_v=\begin{pmatrix}
1 & 1 & 0 & 0 & 0 \\
0 & -1 & 0 & 0 & 0 \\
0 & 1 & 1 & 0 & 0 \\
0 & 0 & 0 & 1 & 0 \\
0 & 0 & 0 & 0 & 1
\end{pmatrix}$

\begin{multline*}
\sigma(v)=\det(\Lambda_v)\det(a_1, a_3, a_4, a_5, a_6)=\\
=\det\begin{pmatrix}
1 & 1 & 0 & 0 & 0 \\
0 & -1 & 0 & 0 & 0 \\
0 & 1 & 1 & 0 & 0 \\
0 & 0 & 0 & 1 & 0 \\
0 & 0 & 0 & 0 & 1
\end{pmatrix} \det\begin{pmatrix}
1 & -1 & 0 & 0 & 0 \\
0 & -1 & 0 & 0 & 0 \\
0 & 0 & 1 & 0 & 0 \\
0 & 0 & 0 & 1 & 0 \\
0 & 0 & 0 & 0 & 1
\end{pmatrix} = (-1)\cdot (-1)=1
\end{multline*}

\medskip

4) $v=(0,0,1,0,0)=\widetilde F_1\cap \widetilde F_2 \cap \widetilde E_2\cap \widetilde E_3\cap \widetilde E_4$

\medskip

$\Lambda_v=\begin{pmatrix}
1 & 0 & 0 & 0 & 0 \\
0 & 1 & 0 & 0 & 0 \\
0 & 0 & 0 & 0 & 1 \\
0 & 0 & 1 & 0 & -1 \\
0 & 0 & 0 & 1 & 1
\end{pmatrix}$ \quad $\Lambda^{-1}_v=\begin{pmatrix}
1 & 0 & 0 & 0 & 0 \\
0 & 1 & 0 & 0 & 0 \\
0 & 0 & 1 & 1 & 0 \\
0 & 0 & -1 & ] & 1 \\
0 & 0 & 1 & 0 & 0
\end{pmatrix}$

\begin{multline*}
\sigma(v)=\det(\Lambda_v)\det(a_1, a_2, a_5, a_6, a_7)=\\
=\det\begin{pmatrix}
1 & 0 & 0 & 0 & 0 \\
0 & 1 & 0 & 0 & 0 \\
0 & 0 & 0 & 0 & 1 \\
0 & 0 & 1 & 0 & -1 \\
0 & 0 & 0 & 1 & 1
\end{pmatrix} \det\begin{pmatrix}
1 & 0 & 0 & 0 & 0 \\
0 & 1 & 0 & 0 & 0 \\
0 & 0 & 0 & 0 & -1 \\
0 & 0 & 1 & 0 & -1 \\
0 & 0 & 0 & 1 & -1
\end{pmatrix} = 1\cdot (-1)=-1
\end{multline*}

\medskip

5) $v=(0,1,1,0,0)=\widetilde F_1\cap \widetilde F_3 \cap \widetilde E_2\cap \widetilde E_3\cap \widetilde E_4$

\medskip

$\Lambda_v=\begin{pmatrix}
1 & 1 & 0 & 0 & 0 \\
0 & -1 & 0 & 0 & 0 \\
0 & 1 & 0 & 0 & 1 \\
0 & 0 & 1 & 0 & -1 \\
0 & 0 & 0 & 1 & 1
\end{pmatrix}$ \quad $\Lambda^{-1}_v=\begin{pmatrix}
1 & 1 & 0 & 0 & 0 \\
0 & -1 & 0 & 0 & 0 \\
0 & 1 & 1 & 1 & 0 \\
0 & -1 & -1 & 0 & 1 \\
0 & 1 & 1 & 0 & 0
\end{pmatrix}$

\begin{multline*}
\sigma(v)=\det(\Lambda_v)\det(a_1, a_3, a_5, a_6, a_7)=\\
=\det\begin{pmatrix}
1 & 1 & 0 & 0 & 0 \\
0 & -1 & 0 & 0 & 0 \\
0 & 1 & 0 & 0 & 1 \\
0 & 0 & 1 & 0 & -1 \\
0 & 0 & 0 & 1 & 1
\end{pmatrix} \det\begin{pmatrix}
1 & -1 & 0 & 0 & 0 \\
0 & -1 & 0 & 0 & 0 \\
0 & 0 & 0 & 0 & -1 \\
0 & 0 & 1 & 0 & -1 \\
0 & 0 & 0 & 1 & -1
\end{pmatrix} = (-1)\cdot 1=-1
\end{multline*}

\medskip

6) $v=(1,0,1,0,0)=\widetilde F_2\cap \widetilde F_3 \cap \widetilde E_2\cap \widetilde E_3\cap \widetilde E_4$

\medskip

$\Lambda_v=\begin{pmatrix}
0 & 1 & 0 & 0 & 0 \\
1 & -1 & 0 & 0 & 0 \\
0 & 1 & 0 & 0 & 1 \\
0 & 0 & 1 & 0 & -1 \\
0 & 0 & 0 & 1 & 1
\end{pmatrix}$ \quad $\Lambda^{-1}_v=\begin{pmatrix}
1 & 1 & 0 & 0 & 0 \\
1 & 0 & 0 & 0 & 0 \\
-1 & 0 & 1 & 1 & 0 \\
1 & 0 & -1 & 0 & 1 \\
-1 & 0 & 1 & 0 & 0
\end{pmatrix}$

\begin{multline*}
\sigma(v)=\det(\Lambda_v)\det(a_2, a_3, a_5, a_6, a_7)=\\
=\det\begin{pmatrix}
0 & 1 & 0 & 0 & 0 \\
1 & -1 & 0 & 0 & 0 \\
0 & 1 & 0 & 0 & 1 \\
0 & 0 & 1 & 0 & -1 \\
0 & 0 & 0 & 1 & 1
\end{pmatrix} \det\begin{pmatrix}
0 & -1 & 0 & 0 & 0 \\
1 & -1 & 0 & 0 & 0 \\
0 & 0 & 0 & 0 & -1 \\
0 & 0 & 1 & 0 & -1 \\
0 & 0 & 0 & 1 & -1
\end{pmatrix} = (-1)\cdot (-1)=1
\end{multline*}

\medskip

7) $v=(0,0,0,1,0)=\widetilde F_1\cap \widetilde F_2 \cap \widetilde E_1\cap \widetilde E_3\cap \widetilde E_4$

\medskip

$\Lambda_v=\begin{pmatrix}
1 & 0 & 0 & 0 & 0 \\
0 & 1 & 0 & 0 & 0 \\
0 & 0 & 1 & 0 & 1 \\
0 & 0 & 0 & 0 & -1 \\
0 & 0 & 0 & 1 & 1
\end{pmatrix}$ \quad $\Lambda^{-1}_v=\begin{pmatrix}
1 & 0 & 0 & 0 & 0 \\
0 & 1 & 0 & 0 & 0 \\
0 & 0 & 1 & 1 & 0 \\
0 & 0 & 0 & 1 & 1 \\
0 & 0 & 0 & -1 & 0
\end{pmatrix}$

\begin{multline*}
\sigma(v)=\det(\Lambda_v)\det(a_1, a_2, a_4, a_6, a_7)=\\
=\det\begin{pmatrix}
1 & 0 & 0 & 0 & 0 \\
0 & 1 & 0 & 0 & 0 \\
0 & 0 & 1 & 0 & 1 \\
0 & 0 & 0 & 0 & -1 \\
0 & 0 & 0 & 1 & 1
\end{pmatrix} \det\begin{pmatrix}
1 & 0 & 0 & 0 & 0 \\
0 & 1 & 0 & 0 & 0 \\
0 & 0 & 1 & 0 & -1 \\
0 & 0 & 0 & 0 & -1 \\
0 & 0 & 0 & 1 & -1
\end{pmatrix} = 1\cdot 1=1
\end{multline*}

\medskip

8) $v=(0,1,0,1,0)=\widetilde F_1\cap \widetilde F_3 \cap \widetilde E_1\cap \widetilde E_3\cap \widetilde E_4$

\medskip

$\Lambda_v=\begin{pmatrix}
1 & 1 & 0 & 0 & 0 \\
0 & -1 & 0 & 0 & 0 \\
0 & 1 & 1 & 0 & 1 \\
0 & 0 & 0 & 0 & -1 \\
0 & 0 & 0 & 1 & 1
\end{pmatrix}$ \quad $\Lambda^{-1}_v=\begin{pmatrix}
1 & 1 & 0 & 0 & 0 \\
0 & -1 & 0 & 0 & 0 \\
0 & 1 & 1 & 1 & 0 \\
0 & 0 & 0 & 1 & 1 \\
0 & 0 & 0 & -1 & 0
\end{pmatrix}$

\begin{multline*}
\sigma(v)=\det(\Lambda_v)\det(a_1, a_3, a_4, a_6, a_7)=\\
=\det\begin{pmatrix}
1 & 1 & 0 & 0 & 0 \\
0 & -1 & 0 & 0 & 0 \\
0 & 1 & 1 & 0 & 1 \\
0 & 0 & 0 & 0 & -1 \\
0 & 0 & 0 & 1 & 1
\end{pmatrix} \det\begin{pmatrix}
1 & -1 & 0 & 0 & 0 \\
0 & -1 & 0 & 0 & 0 \\
0 & 0 & 1 & 0 & -1 \\
0 & 0 & 0 & 0 & -1 \\
0 & 0 & 0 & 1 & -1
\end{pmatrix} = (-1)\cdot (-1)=1
\end{multline*}

\medskip

9) $v=(1,0,0,1,0)=\widetilde F_2\cap \widetilde F_3 \cap \widetilde E_1\cap \widetilde E_3\cap \widetilde E_4$

\medskip

$\Lambda_v=\begin{pmatrix}
0 & 1 & 0 & 0 & 0 \\
1 & -1 & 0 & 0 & 0 \\
0 & 1 & 1 & 0 & 1 \\
0 & 0 & 0 & 0 & -1 \\
0 & 0 & 0 & 1 & 1
\end{pmatrix}$ \quad $\Lambda^{-1}_v=\begin{pmatrix}
1 & 1 & 0 & 0 & 0 \\
1 & 0 & 0 & 0 & 0 \\
-1 & 0 & 1 & 1 & 0 \\
0 & 0 & 0 & 1 & 1 \\
0 & 0 & 0 & -1 & 0
\end{pmatrix}$

\begin{multline*}
\sigma(v)=\det(\Lambda_v)\det(a_2, a_3, a_4, a_6, a_7)=\\
=\det\begin{pmatrix}
0 & 1 & 0 & 0 & 0 \\
1 & -1 & 0 & 0 & 0 \\
0 & 1 & 1 & 0 & 1 \\
0 & 0 & 0 & 0 & -1 \\
0 & 0 & 0 & 1 & 1
\end{pmatrix} \det\begin{pmatrix}
0 & -1 & 0 & 0 & 0 \\
1 & -1 & 0 & 0 & 0 \\
0 & 0 & 1 & 0 & -1 \\
0 & 0 & 0 & 0 & -1 \\
0 & 0 & 0 & 1 & -1
\end{pmatrix} = (-1)\cdot 1=-1
\end{multline*}

\medskip

10) $v=(0,0,0,0,1)=\widetilde F_1\cap \widetilde F_2 \cap \widetilde E_1\cap \widetilde E_2\cap \widetilde E_4$

\medskip

$\Lambda_v=\begin{pmatrix}
1 & 0 & 0 & 0 & 0 \\
0 & 1 & 0 & 0 & 0 \\
0 & 0 & 1 & 0 & 1 \\
0 & 0 & 0 & 1 & -1 \\
0 & 0 & 0 & 0 & 1
\end{pmatrix}$ \quad $\Lambda^{-1}_v=\begin{pmatrix}
1 & 0 & 0 & 0 & 0 \\
0 & 1 & 0 & 0 & 0 \\
0 & 0 & 1 & 0 & -1 \\
0 & 0 & 0 & 1 & 1 \\
0 & 0 & 0 & 0 & 1
\end{pmatrix}$

\begin{multline*}
\sigma(v)=\det(\Lambda_v)\det(a_1, a_2, a_4, a_5, a_7)=\\
=\det\begin{pmatrix}
1 & 0 & 0 & 0 & 0 \\
0 & 1 & 0 & 0 & 0 \\
0 & 0 & 1 & 0 & 1 \\
0 & 0 & 0 & 1 & -1 \\
0 & 0 & 0 & 0 & 1
\end{pmatrix} \det\begin{pmatrix}
1 & 0 & 0 & 0 & 0 \\
0 & 1 & 0 & 0 & 0 \\
0 & 0 & 1 & 0 & -1 \\
0 & 0 & 0 & 1 & -1 \\
0 & 0 & 0 & 0 & -1
\end{pmatrix} = 1\cdot (-1)=-1
\end{multline*}

\medskip

11) $v=(0,1,0,0,1)=\widetilde F_1\cap \widetilde F_3 \cap \widetilde E_1\cap \widetilde E_2\cap \widetilde E_4$

\medskip

$\Lambda_v=\begin{pmatrix}
1 & 1 & 0 & 0 & 0 \\
0 & -1 & 0 & 0 & 0 \\
0 & 1 & 1 & 0 & 1 \\
0 & 0 & 0 & 1 & -1 \\
0 & 0 & 0 & 0 & 1
\end{pmatrix}$ \quad $\Lambda^{-1}_v=\begin{pmatrix}
1 & 1 & 0 & 0 & 0 \\
0 & -1 & 0 & 0 & 0 \\
0 & 1 & 1 & 0 & -1 \\
0 & 0 & 0 & 1 & 1 \\
0 & 0 & 0 & 0 & 1
\end{pmatrix}$

\begin{multline*}
\sigma(v)=\det(\Lambda_v)\det(a_1, a_3, a_4, a_5, a_7)=\\
=\det\begin{pmatrix}
1 & 1 & 0 & 0 & 0 \\
0 & -1 & 0 & 0 & 0 \\
0 & 1 & 1 & 0 & 1 \\
0 & 0 & 0 & 1 & -1 \\
0 & 0 & 0 & 0 & 1
\end{pmatrix} \det\begin{pmatrix}
1 & -1 & 0 & 0 & 0 \\
0 & -1 & 0 & 0 & 0 \\
0 & 0 & 1 & 0 & -1 \\
0 & 0 & 0 & 1 & -1 \\
0 & 0 & 0 & 0 & -1
\end{pmatrix} = (-1)\cdot 1=-1
\end{multline*}

\medskip

12) $v=(1,0,0,0,1)=\widetilde F_2\cap \widetilde F_3 \cap \widetilde E_1\cap \widetilde E_2\cap \widetilde E_4$

\medskip

$\Lambda_v=\begin{pmatrix}
0 & 1 & 0 & 0 & 0 \\
1 & -1 & 0 & 0 & 0 \\
0 & 1 & 1 & 0 & 1 \\
0 & 0 & 0 & 1 & -1 \\
0 & 0 & 0 & 0 & 1
\end{pmatrix}$ \quad $\Lambda^{-1}_v=\begin{pmatrix}
1 & 1 & 0 & 0 & 0 \\
1 & 0 & 0 & 0 & 0 \\
-1 & 0 & 1 & 0 & -1 \\
0 & 0 & 0 & 1 & 1 \\
0 & 0 & 0 & 0 & 1
\end{pmatrix}$

\begin{multline*}
\sigma(v)=\det(\Lambda_v)\det(a_2, a_3, a_4, a_5, a_7)=\\
=\det\begin{pmatrix}
0 & 1 & 0 & 0 & 0 \\
1 & -1 & 0 & 0 & 0 \\
0 & 1 & 1 & 0 & 1 \\
0 & 0 & 0 & 1 & -1 \\
0 & 0 & 0 & 0 & 1
\end{pmatrix} \det\begin{pmatrix}
0 & -1 & 0 & 0 & 0 \\
1 & -1 & 0 & 0 & 0 \\
0 & 0 & 1 & 0 & -1 \\
0 & 0 & 0 & 1 & -1 \\
0 & 0 & 0 & 0 & -1
\end{pmatrix} = (-1)\cdot (-1)=1
\end{multline*}

So if a genus $\varphi$ is rigid on $\widetilde L(2,3)$, then its exponent $f(x)$ satisfies the rigidity equation

 \begin{multline*}
    c=\frac{1}{f(x_1)f(x_2)f(x_3)f(x_4)f(x_5)}-\frac{1}{f(x_1+x_2)f(x_1)f(x_3-x_1)f(x_4)f(x_5)}+\\
    +\frac{1}{f(x_1+x_2)f(-x_2)f(x_2+x_3)f(x_4)f(x_5)}-\frac{1}{f(x_1)f(x_2)f(x_3+x_4)f(x_5-x_3)f(x_3)}-\\
    -\frac{1}{f(x_1+x_2)f(-x_2)f(x_2+x_3+x_4)f(x_5-x_2-x_3)f(x_2+x_3)}+\\
    +\frac{1}{f(x_1+x_2)f(x_1)f(x_3+x_4-x_1)f(x_1-x_3+x_5)f(x_3-x_1)}+\\
    +\frac{1}{f(x_1)f(x_2)f(x_3+x_4)f(x_4+x_5)f(-x_4)}+\\
    +\frac{1}{f(x_1+x_2)f(-x_2)f(x_2+x_3+x_4)f(x_4+x_5)f(-x_4)}-\\
    -\frac{1}{f(x_1+x_2)f(x_1)f(-x_1+x_3+x_4)f(x_4+x_5)f(-x_4)}-
    \end{multline*}
    \begin{multline*}
    -\frac{1}{f(x_1)f(x_2)f(x_3-x_5)f(x_4+x_5)f(x_5)}-\\
    -\frac{1}{f(x_1+x_2)f(-x_2)f(x_2+x_3-x_5)f(x_4+x_5)f(x_5)}+\\
    +\frac{1}{f(x_1+x_2)f(x_1)f(-x_1+x_3-x_5)f(x_4+x_5)f(x_5)}    
\end{multline*}

If $\varphi$ is also rigid on $S^6$ and $\varphi([S^6])=0$, then by Theorem \ref{s6} we have  $f(x)=e^{\alpha x}g(x)$ for some odd power series $g(x)=x+\ldots$ 

Notice that the sum of arguments in each denominator is equal to $x_1+x_2+x_3+x_4+x_5$ (in fact, the sum of weights does not depend on fixed point for any $SU$-manifold, see \cite{kr74}). Hence we have

\begin{multline*}
    c=e^{-\alpha(x_1+x_2+x_3+x_4+x_5)}\Bigl(\frac{1}{g(x_1)g(x_2)g(x_3)g(x_4)g(x_5)}-\\
    -\frac{1}{g(x_1+x_2)g(x_1)g(x_3-x_1)g(x_4)g(x_5)}+\\
    +\frac{1}{g(x_1+x_2)g(-x_2)g(x_2+x_3)g(x_4)g(x_5)}-\frac{1}{g(x_1)g(x_2)g(x_3+x_4)g(x_5-x_3)g(x_3)}-\\
    -\frac{1}{g(x_1+x_2)g(-x_2)g(x_2+x_3+x_4)g(x_5-x_2-x_3)g(x_2+x_3)}+\\
    +\frac{1}{g(x_1+x_2)g(x_1)g(x_3+x_4-x_1)g(x_1-x_3+x_5)g(x_3-x_1)}+\\
    +\frac{1}{g(x_1)g(x_2)g(x_3+x_4)g(x_4+x_5)g(-x_4)}+\\
    +\frac{1}{g(x_1+x_2)g(-x_2)g(x_2+x_3+x_4)g(x_4+x_5)g(-x_4)}-\\
    -\frac{1}{g(x_1+x_2)g(x_1)g(-x_1+x_3+x_4)g(x_4+x_5)g(-x_4)}-\\
    -\frac{1}{g(x_1)g(x_2)g(x_3-x_5)g(x_4+x_5)g(x_5)}-\\
    -\frac{1}{g(x_1+x_2)g(-x_2)g(x_2+x_3-x_5)g(x_4+x_5)g(x_5)}+\\
    +\frac{1}{g(x_1+x_2)g(x_1)g(-x_1+x_3-x_5)g(x_4+x_5)g(x_5)}\Bigr)
\end{multline*}

\medskip

So setting $x_1+x_2+x_3+x_4+x_5=0$ we obtain the same rigidity equation on $g(x)$:

\begin{multline*}
    c=\frac{1}{g(x_1)g(x_2)g(x_3)g(x_4)g(x_5)}-\frac{1}{g(x_1+x_2)g(x_1)g(x_3-x_1)g(x_4)g(x_5)}+\\
    +\frac{1}{g(x_1+x_2)g(-x_2)g(x_2+x_3)g(x_4)g(x_5)}-\frac{1}{g(x_1)g(x_2)g(x_3+x_4)g(x_5-x_3)g(x_3)}-\\
    -\frac{1}{g(x_1+x_2)g(-x_2)g(x_2+x_3+x_4)g(x_5-x_2-x_3)g(x_2+x_3)}+
    \end{multline*}
    \begin{multline*}
    +\frac{1}{g(x_1+x_2)g(x_1)g(x_3+x_4-x_1)g(x_1-x_3+x_5)g(x_3-x_1)}+\\
    +\frac{1}{g(x_1)g(x_2)g(x_3+x_4)g(x_4+x_5)g(-x_4)}+\\
    +\frac{1}{g(x_1+x_2)g(-x_2)g(x_2+x_3+x_4)g(x_4+x_5)g(-x_4)}-\\
    -\frac{1}{g(x_1+x_2)g(x_1)g(-x_1+x_3+x_4)g(x_4+x_5)g(-x_4)}-\\
    -\frac{1}{g(x_1)g(x_2)g(x_3-x_5)g(x_4+x_5)g(x_5)}-\\
    -\frac{1}{g(x_1+x_2)g(-x_2)g(x_2+x_3-x_5)g(x_4+x_5)g(x_5)}+\\
    +\frac{1}{g(x_1+x_2)g(x_1)g(-x_1+x_3-x_5)g(x_4+x_5)g(x_5)}
\end{multline*}

\vspace{3ex}

1) Consider the limit of the sum of summands 2 and 3 as $x_1+x_2=t\to 0$.

\begin{multline*}
\frac{1}{g(x_1+x_2)g(-x_2)g(x_2+x_3)g(x_4)g(x_5)}-\frac{1}{g(x_1+x_2)g(x_1)g(x_3-x_1)g(x_4)g(x_5)}=\\
=\frac{1}{g(x_4)g(x_5)}\frac{1}{g(t)}\bigl(\frac{1}{g(x_1-t)g(x_3-(x_1-t))}-\frac{1}{g(x_1)g(x_3-x_1)}\bigr)=\\
=\frac{1}{g(x_4)g(x_5)}\frac{-t}{g(t)}\frac{1}{-t}\bigl(\frac{1}{g(x_1-t)g(x_3-(x_1-t))}-\frac{1}{g(x_1)g(x_3-x_1)}\bigr)\to\\
\to -\frac{1}{g(x_4)g(x_5)}\bigl(\frac{1}{g(x)g(x_3-x)}\bigr)'|_{x=x_1}=\frac{1}{g(x_4)g(x_5)}\bigl(\frac{g'(x_1)g(x_3-x_1)-g(x_1)g'(x_3-x_1)}{g(x_1)^2 g(x_3-x_1)^2}\bigr)
\end{multline*}

\vspace{3ex}

2) Consider the limit of the sum of summands 8 and 9 as $x_1+x_2=t\to 0$ under the additional restrictions $x_3+x_4=0$, $x_5=0$.

\begin{multline*}
\frac{1}{g(x_1+x_2)g(-x_2)g(x_2+x_3+x_4)g(x_4+x_5)g(-x_4)}-\\
    -\frac{1}{g(x_1+x_2)g(x_1)g(-x_1+x_3+x_4)g(x_4+x_5)g(-x_4)}=\\
    =\frac{1}{g(x_4)g(-x_4)}\frac{-t}{g(t)}\frac{1}{-t}\bigl(\frac{1}{g(x_1-t)g(-(x_1-t))}-\frac{1}{g(x_1)g(-x_1)}\bigr)\to\\
    \to-\frac{1}{g(x_4)g(-x_4)}\bigl(\frac{1}{g(x)g(-x)}\bigr)'|_{x=x_1}=\frac{1}{g(x_4)^2}\bigl(\frac{1}{g(x)^2}\bigr)'|_{x=x_1}=\frac{2g'(x_1)}{g(x_4)^2g(x_1)^3}
    \end{multline*}

\vspace{3ex}

3) Consider the limit of the sum of summands 5 and 6 as $x_1+x_2=t\to 0$ under the additional restrictions $x_3+x_4=0$, $x_5=0$.

\begin{multline*}
\frac{1}{g(x_1+x_2)g(x_1)g(x_3+x_4-x_1)g(x_1-x_3+x_5)g(x_3-x_1)}-
\end{multline*}
\begin{multline*}
-\frac{1}{g(x_1+x_2)g(-x_2)g(x_2+x_3+x_4)g(x_5-x_2-x_3)g(x_2+x_3)}=\\
=\frac{1}{g(x_1+x_2)g(x_1)g(-x_1)g(x_1-x_3)g(x_3-x_1)}-\\
-\frac{1}{g(x_1+x_2)g(-x_2)g(x_2)g(-x_2-x_3)g(x_2+x_3)}=\\
=\frac{t}{g(t)}\frac{1}{t}\bigl(\frac{1}{g(-x_2+t)g(-(-x_2+t))g(-x_3-x_2+t)g(x_3-(-x_2+t))}-\\
-\frac{1}{g(-x_2)g(x_2)g(-x_2-x_3)g(x_2+x_3)}\bigr)\to \bigl(\frac{1}{g(x)g(-x)g(-x_3+x)g(x_3-x)}\bigr)'|_{x=-x_2}=\\
=\bigl(\frac{1}{g(x)^2 g(x-x_3)^2}\bigr)'|_{x=-x_2}=\\
-\frac{2g'(-x_2)g(-x_2)g(x_2+x_3)^2+2g(x_2)^2g(-x_2-x_3)g'(-x_2-x_3)}{g(x_2)^4g(x_2+x_3)^4}=\\
=2\frac{g'(x_2)g(x_2+x_3)+g(x_2)g'(x_2+x_3)}{g(x_2)^3g(x_2+x_3)^3}
\end{multline*}

\vspace{3ex}

4) Consider the limit of the sum of summands 1 and 10 as $x_3+x_4=s\to 0$ under the additional restriction $x_1+x_2=0$ (as $x_1+x_2+x_3+x_4+x_5=0$ we also have $x_5=-s$).

    \begin{multline*}
    \frac{1}{g(x_1)g(x_2)g(x_3)g(x_4)g(x_5)}-\frac{1}{g(x_1)g(x_2)g(x_3-x_5)g(x_4+x_5)g(x_5)}=\\
    =\frac{1}{g(x_1)g(x_2)g(x_3)g(-s)}(\frac{1}{g(x_3+s)}-\frac{1}{g(x_3-s)})\to \frac{2g'(x_3)}{g(x_1)g(x_2)g(x_3)^3}
    \end{multline*}

    \vspace{3ex}

    5) Consider the limit of the sum of summands 4 and 7 as $x_3+x_4=s\to 0$ under the additional restriction $x_1+x_2=0$ (as $x_1+x_2+x_3+x_4+x_5=0$ we also have $x_5=-s$).

    \begin{multline*}
    -\frac{1}{g(x_1)g(x_2)g(x_3+x_4)g(x_5-x_3)g(x_3)}+\frac{1}{g(x_1)g(x_2)g(x_3+x_4)g(x_4+x_5)g(-x_4)}=\\
    =\frac{1}{g(x_1)g(x_2)g(x_3)g(s)}(\frac{1}{g(x_3+s)}-\frac{1}{g(x_3-s)})\to -\frac{2g'(x_3)}{g(x_1)g(x_2)g(x_3)^3}
    \end{multline*}
    
    \vspace{3ex}

6) Consider the limit of the sum of summands 11 and 12 as $x_1+x_2=t\to 0$.

    \begin{multline*}
    -\frac{1}{g(x_1+x_2)g(-x_2)g(x_2+x_3-x_5)g(x_4+x_5)g(x_5)}+\\
    +\frac{1}{g(x_1+x_2)g(x_1)g(-x_1+x_3-x_5)g(x_4+x_5)g(x_5)}=\\
    =\frac{1}{g(x_4+x_5)g(x_5)g(t)}\bigl(\frac{1}{g(-x_2+t)g(-(-x_2+t)+x_3-x_5)}-\frac{1}{g(-x_2)g(x_2+x_3-x_5)}\bigr)\to\\
    \to\frac{1}{g(x_4+x_5)g(x_5)}\bigl(\frac{1}{g(x)g(-x+x_3-x_5)}\bigr)'|_{x=-x_2}=
    \end{multline*}
    \begin{multline*}
    =-\frac{1}{g(x_4+x_5)g(x_5)}\bigl(\frac{g'(x_2)g(x_2+x_3-x_5)+g'(x_2+x_3-x_5)g(x_2)}{g(x_2)^2g(x_2+x_3-x_5)^2}\bigr)=\\
    =-\frac{1}{g(x_4+x_5)g(x_5)}(\frac{g'(x_1)g(-x_1+x_3-x_5)-g'(-x_1+x_3-x_5)g(x_1)}{g(x_1)^2g(-x_1+x_3-x_5)^2})
    \end{multline*}

    \vspace{3ex}

    Now consider the limit of the sum of the results of 1) and 6) as $x_3+x_4=s\to 0$ (and hence $x_5=-s\to 0$).

    \begin{multline*}
    \frac{1}{g(x_4)g(x_5)}\bigl(\frac{g'(x_1)g(x_3-x_1)-g(x_1)g'(x_3-x_1)}{g(x_1)^2 g(x_3-x_1)^2}\bigr)-\\
    -\frac{1}{g(x_4+x_5)g(x_5)}\bigl(\frac{g'(x_1)g(-x_1+x_3-x_5)-g'(-x_1+x_3-x_5)g(x_1)}{g(x_1)^2g(-x_1+x_3-x_5)^2}\bigr)=\\
    =\frac{s}{g(s)}\frac{1}{g(x_1)^2}\frac{1}{s}\bigl(\frac{g'(x_1)g(-x_1+x_3+s)-g'(-x_1+x_3+s)g(x_1)}{g(-x_1+x_3+s)^2g(-x_3)}-\\
    -\frac{g'(x_1)g(x_3-x_1)-g(x_1)g'(x_3-x_1)}{g(x_3-x_1)^2g(-(x_3-s))}\bigr)\to\\
    \to\frac{1}{g(x_1)^2}\bigl(\frac{g'(x_1)g(-x_1+x)-g'(-x_1+x)g(x_1)}{g(-x_1+x)^2g(-x)}\bigr)'|_{x=x_3}=\\
    =\frac{1}{g(x_1)^2g(x_3-x_1)^4g(x_3)^2}\Bigl(-g(x_3)g(x_3-x_1)^2\bigl(g'(x_1)g'(x_3-x_1)-g''(x_3-x_1)g(x_1)\bigr)+\\
    +\bigl(g'(x_3)g(x_3-x_1)^2+2g(x_3-x_1)g'(x_3-x_1)g(x_3)\bigr)\bigl(g'(x_1)g(x_3-x_1)-g'(x_3-x_1)g(x_1)\bigr)\Bigr)=\\
    =\frac{1}{g(x_1)^2g(x_3-x_1)^4g(x_3)^2}\Bigl(g(x_1)g(x_3)g(x_3-x_1)^2g''(x_3-x_1)-\\
    -g(x_3)g(x_3-x_1)^2g'(x_1)g'(x_3-x_1)+g'(x_1)g'(x_3)g(x_3-x_1)^3+2g'(x_1)g'(x_3-x_1)g(x_3-x_1)^2g(x_3)-\\
    -g(x_1)g'(x_3)g'(x_3-x_1)g(x_3-x_1)^2-2g(x_1)g(x_3)g(x_3-x_1)(g'(x_3-x_1)^2)\Bigr)=\\
    =\frac{g'(x_1)g'(x_3-x_1)}{g(x_1)^2g(x_3)g(x_3-x_1)^2} +\frac{g''(x_3-x_1)}{g(x_1)g(x_3)g(x_3-x_1)^2}-\frac{2g'(x_3-x_1)^2}{g(x_1)g(x_3)g(x_3-x_1)^3}+\\
    +\frac{g'(x_1)}{g(x_1)^2g(x_3)^2g(x_3-x_1)}-\frac{g'(x_3-x_1)g'(x_3)}{g(x_+)g(x_3-x_1)^2g(x_3)^2}\bigr)
    \end{multline*}

\vspace{3ex}

    Adding all summands together we obtain that the rigidity equation under the restrictions $x_1+x_2=0$, $x_3+x_4=0$, $x_5=0$ has the following form on the variables $x_1=-x_2=x$ and $x_3=-x_4=y$:

    \begin{multline*}
        c=\frac{2g'(x)}{g(y)^2g(x)^3}-2\frac{g'(x)}{g(x)^3g(y-x)^2}+2\frac{g'(y-x)}{g(x)^2g(y-x)^3}+\frac{g'(x)g'(y-x)}{g(x)^2g(y)g(y-x)^2}+\\
        +\frac{g''(y-x)}{g(y)g(x)g(y-x)^2}-\frac{2g'(y-x)^2}{g(x)g(y)g(y-x)^3}+\frac{g'(x)g'(y)}{g(x)^2g(y)^2g(y-x)}-\frac{g'(y-x)g'(y)}{g(y-x)^2g(x)g(y)^2}
    \end{multline*}

    \bigskip

    Setting $x=y+z$ we get

    \begin{multline*}
        c=\frac{2g'(y+z)}{g(y)^2g(y+z)^3}-\frac{2g'(y+z)}{g(y+z)^3g(z)^2}-\frac{2g'(z)}{g(y+z)^2g(z)^3}+\frac{g'(y+z)g'(z)}{g(y+z)^2g(y)g(z)^2}-\\
        -\frac{g''(z)}{g(y)g(y+z)g(z)^2}+\frac{2g'(z)^2}{g(y+z)g(y)g(z)^3}-\frac{g'(y+z)g'(y)}{g(y+z)^2g(y)^2g(z)}-\frac{g'(z)g'(y)}{g(z)^2g(y+z)g(y)^2}
    \end{multline*}

\bigskip

    After symmetrisation we obtain

    \begin{multline*}
        2c=-\frac{2g'(z)}{g(y+z)^2g(z)^3}-\frac{2g'(y)}{g(y+z)^2g(y)^3}-\frac{g''(z)}{g(y)g(y+z)g(z)^2}-\frac{g''(y)}{g(z)g(y+z)g(y)^2}+\\
        +\frac{2g'(z)^2}{g(y+z)g(y)g(z)^3}+\frac{2g'(y)^2}{g(y+z)g(z)g(y)^3}-\frac{2g'(z)g'(y)}{g(z)^2g(y+z)g(y)^2}
    \end{multline*}

 \begin{multline*}
        2cg(y+z)^2g(y)^3g(z)^3+2g'(z)g(y)^3+2g'(y)g(z)^3=g(y+z)\Bigl(-\bigl(g''(z)g(y)^2g(z)+\\
        +g''(y)g(z)^2g(y)\bigr)+2\bigl(g'(z)^2g(y)^2+g'(y)^2g(z)^2-g'(y)g'(z)g(y)g(z)\bigr)\Bigr)
    \end{multline*}

\bigskip

Expand our power series on $z$ up to $z^3$:

\medskip

    $g(y+z)=g(y)+g'(y)z+g''(y)\frac{z^2}{2}+g'''(y)\frac{z^3}{6}+\ldots$

\medskip

    $g(z)=z+\frac{a}{3}z^3+\ldots$

    \medskip

    $g'(z)=1+az^2+\ldots$

    \medskip

    $g''(z)=2az+bz^3+\ldots$

    \medskip

    $g'(z)^2=1+2az^2+\ldots$

    \medskip

    $g(z)^2=z^2+\ldots$

    \medskip

    $g^3(z)=z^3+\ldots$

    \bigskip

    Below for the sake of brevity we will denote $g(y)$ by $g$ and similarly for its derivatives.

\smallskip

    Then our restricted rigidity equation has the following form up to $z^3$:

    \begin{multline*}
        2cg^5z^3+2(1+az^2)g^3+2g'z^3=(g+g'z+g''\frac{z^2}{2}+g'''\frac{z^3}{6})\Bigl(-\bigl((2az+bz^3)(z+\frac{a}{3}z^3)g^2+\\
        +g''gz^2\bigr)+2\bigl((1+2az^2)g^2+(g')^2z^2-g'g(1+az^2)(z+\frac{a}{3}z^3)\bigr)\Bigr)=(g+g'z+g''\frac{z^2}{2}+g'''\frac{z^3}{6})\times\\
        \times\Bigl(-\bigl(2az^2g^2+g''gz^2\bigr) +2\bigl((1+2az^2)g^2+(g')^2z^2-g'g(z+\frac{4a}{3}z^3)\bigr)\Bigr)=\\
        =(g+g'z+g''\frac{z^2}{2}+g'''\frac{z^3}{6})(2g^2-2gg'z+z^2(-2ag^2-g''g+4ag^2+2(g')^2)-g'g\frac{8a}{3}z^3)
    \end{multline*}

\smallskip
    Then the coefficient on the $z^3$ term is

    $$2cg^5+2g'=-\frac{8a}{3}g'g^2+2ag^2g'-gg'g''+2(g')^3-gg'g''+\frac{g^2g'''}{3}$$

\medskip
    So we obtain the differential equation

    $$2cg^5+2g'=-2g''g'g-\frac{2a}{3}g^2g'+2(g')^3+\frac{g'''g^2}{3}$$

    Since all terms except for $2cg^5$ are even we immediately obtain that $c=0$ (as predicted by Theorem \ref{zero}).

    $$g'(6+6gg''+2ag^2-6(g')^2)=g'''g^2$$

\bigskip

    As $(g')^2$ is even and $g(y)=y+\ldots$, we can set $(g')^2=1+\sum_{k\ge 1}a_k g^{2k}$ for some $a_k$. Note that $a_1=2a$, since $(g')^2=1+2ay^2+\ldots$
    
Differentiating we get $g'g''=2agg'+\sum_{k\ge 2}ka_kg'g^{2k-1}$. As the power series $g'$ is invertible, we obtain $g''=2ag+\sum_{k\ge 2}ka_k g^{2k-1}$.

Differentiating again we get $g'''=2ag'+\sum_{k\ge 2}k(2k-1)a_k g^{2k-2}g'$.

\smallskip

Substituting in the differential equation we obtain

$$g'(6+12ag^2+6\sum_{k\ge2}ka_kg^{2k}+2ag^2-6-12ag^2-6\sum_{k\ge 2}a_kg^{2k})=g'(2ag^2+\sum_{k\ge 2}k(2k-1)a_k g^{2k})$$

$$6\sum_{k\ge 2}(k-1)a_kg^{2k}=\sum_{k\ge 2}k(2k-1)a_kg^{2k}$$

Comparing coefficients on $g^{2k}$ we get $a_{k}=0$ for $k\ge 3$. Then we obtain the equation $(g')^2=1+2ag^2+a_2g^4$. Therefore $g$ is the elliptic sine and $f(x)=e^{\alpha x} \mathrm{sn}(x)$.
\end{proof}

\end{document}